\documentclass[11pt]{article}

\usepackage{xcolor}

\usepackage{amssymb}

\usepackage{amsmath}

\usepackage{amsthm}

\begin{document}

\parskip 1ex            % sets the gap between paragraphs

\parindent 5ex		% an ex is the width of an ex

\renewcommand{\baselinestretch}{1.2}

\newtheorem{theorem}{Theorem}[section]

\newtheorem{lemma}[theorem]{Lemma}

\newtheorem{conjecture}[theorem]{Conjecture}

\newtheorem{proposition}[theorem]{Proposition}

\newtheorem{corollary}[theorem]{Corollary}

\newtheorem{claim}[theorem]{Claim}

\newtheorem{property}{Property}

\newtheorem{definition}{Definition}

\newcommand\mc{\mathcal}

\newcommand\up{\uparrow}

\newcommand{\of}[1]{\left( #1 \right) }

\def\a{\alpha} \def\b{\beta} \def\d{\delta} \def\D{\Delta}

\def\e{\epsilon} \def\f{\phi} \def\F{{\Phi}} \def\vp{\varphi} \def\g{\gamma}

\def\G{\Gamma} \def\i{\iota} \def\k{\e/2} \def\K{\Kappa}

\def\z{\zeta} \def\th{\theta} \def\TH{\Theta} \def\Th{\Theta}  \def\l{\lambda}

\def\La{\Lambda} \def\m{\mu} \def\n{\nu} \def\p{\pi}

\def\r{\rho} \def\R{\Rho} \def\s{\sigma} \def\S{\Sigma}

\def\t{\tau} \def\om{\omega} \def\OM{\Omega} \def\Om{\Omega}\def\U{\Upsilon} \def\l{\ell}

\def\P{\mathbb{P}}

\def\E{\mathbb{E}}

\newcommand\mb{\mathbb}

\newcommand{\tcb}[1]{\textcolor{blue}{#1}}

\title{A note on the random greedy independent set algorithm}

\author{
Patrick Bennett\thanks{Department of Mathematical Sciences, Carnegie Mellon
University, Pittsburgh, PA 15213, USA. Email: {\tt ptbennet@andrew.cmu.edu}.
Research supported in part by NSF grants DMS-1001638 and DMS-1362785.}
\and
Tom Bohman\thanks{Department of Mathematical Sciences, Carnegie Mellon
University, Pittsburgh, PA 15213, USA. Email: {\tt tbohman@math.cmu.edu}.
Research supported in part by NSF grant DMS-1001638.}
}

\date{}

\maketitle

\begin{abstract}

Let $r\ge 3$ be a fixed constant and let 
$ {\mathcal H}$ be an $r$-uniform,
$D$-regular hypergraph on $N$ vertices.  
Assume further that $ D > N^\e $ for some $ \e>0 $.
Consider the random greedy algorithm for forming an independent
set in $ \mc{H}$.  An independent set is chosen at random
by iteratively 
choosing vertices at random to be in the 
independent set.  At each step we chose a vertex uniformly at random
from the collection of vertices that could be added to the independent 
set (i.e. the collection of vertices $v$ with the property that $v$ is not in the current independent set $I$ and $ I \cup \{v\}$ contains no edge in $ \mc{H}$).  Note that this process
terminates at a maximal subset of vertices with the property
that this set contains no edge of $ \mc{H} $; that is, the process
terminates at a maximal independent set.

We prove that if $ \mc{H}$ satisfies certain degree and codegree conditions then 
there are
$ \Omega\left( N \cdot ( (\log N) / D )^{\frac{1}{r-1}} \right) $
vertices in the independent set produced by the random greedy algorithm with 
high probability.  This result generalizes 
a lower bound on the number of steps in the $ H$-free process
due to Bohman and Keevash and produces objects of 
interest in additive combinatorics.
\end{abstract}

\section{Introduction}

Consider the random greedy algorithm for finding a maximal independent
set in a hypergraph.  Let $ {\mathcal H} $ be a hypergraph on vertex
set $V$. (I.e. ${\mathcal H}$ is a collection of subsets of $V$.  The sets in this
collection are the {\em edges} of $\mc{H}$).  An {\em independent set} in $\mc{H}$ 
is a set $ I \subseteq V $ such that
$I$ contains no edge of $ {\mathcal H} $.  The random greedy algorithm forms a maximal
independent set in $ {\mathcal H}$ by iteratively choosing vertices at random to be vertices
in the independent set.  To be precise, we begin with $ {\mathcal H}(0) = {\mathcal H}$, $ V(0) = V$ 
and $ I(0) = \emptyset $.  Given independent set $I(i)$ and 
hypergraph $ \mc{H}(i) $ on vertex set $V(i)$, a vertex $v \in V(i)$ is chosen uniformly at random
and added to $ I(i)$ to form $ I(i+1) $.  The vertex set $V(i+1)$ is set equal to
$ V(i)$ less $v$ and every vertex $u$ such that 
the pair $ \{u,v\} $ is an edge of $ \mc{H}(i) $.  Finally the
hypergraph $ \mc{H}(i+1) $ is formed from $ \mc{H}(i) $  by 
\begin{enumerate}
\item removing $v$ from every edge in $ \mc{H}(i) $ that contains $v$ and at least 2 other vertices, 
\, and
\item removing every edge that contains a vertex $u$ 
such that the pair $ \{u,v\} $ is an edge of $ \mc{H}(i) $.
\end{enumerate}
The process terminates when $ V(i) $ is empty.  At this point $I(i)$ is a maximal independent set
in $\mc{H}$.

A number of problems in combinatorics can be stated in terms of
maximal independent sets in hypergraphs.  In some of these situations, the 
random greedy algorithm produces such an
independent set with desirable properties.  For example, the best known lower 
bounds on the Tur\'an numbers of some bipartite graphs as well as the best known
lower bound on the off-diagonal graph Ramsey numbers $ R(s,t) $ (where $s \ge 3$ is fixed and
$t$ is large) are given by objects produced by this
algorithm.  In these two cases the objects of interest are produced 
by an instance of the random greedy
independent set algorithm known as the $H$-free process.  Here
we let $H$ be a fixed 2-balanced graph (e.g. $ K_\l $) and consider
the hypergraph $\mc{H}_H$ that has vertex set $V = \binom{[n]}{2} $, i.e.
the edge set of the complete graph $K_n$, and 
edge set consisting of all copies of $H$ in $K_n$.    Note that in this 
context the random greedy independent set algorithm produces a graph on vertex set
$[n]$ (i.e. a 
subset of $ \binom{[n]}{2}$) that contains no copy of the graph $H$.
Bohman and Keevash \cite{BK} gave an analysis of the 
$H$-free process for an arbitrary 2-balanced graph $H$ that gives a lower bound
on the number of steps in the process.  In this note we
extend that result to a more general setting.  This generalization includes
natural hypergraph variants of the $H$-free process as well as some
processes that are of interest in number theory.

Following the intuition that guides the earlier work on the $H$-free process,
our study of the random greedy independent set algorithm on a general
$D$-regular, $r$-uniform hypergraph $ \mc{H}$ on vertex set $V$ is guided by the following
question: 
\begin{quote}
To what extent does the independent set $ I(i) $ resemble a random subset $S(i)$ of $V$
chosen by simply taking $ Pr( v \in S(i)) = i/N$, independently, for all $ v \in V $?
\end{quote}
Of course, if $i$ is large enough then the set 
$ S(i)$ should contain many edges of $ \mc{H}$ while $ I(i) $ contains none.  But are 
these sets similar with respect to other statistics?  Consider, for example, the set of vertices 
$ V(i) $, which is the set of vertices that remain eligible for inclusion in the independent set.  A 
vertex $w$ (that does not lie in $ I(i)$ itself) is in this set if 
there is no edge $ e \in \mc{H} $ such that $ w \in e $ and 
$ e \setminus \{w\} \subseteq I(i) $.   If $ I(i) $ resembles
$S(i)$ then the number of vertices that have this property should be roughly
\[ |V| \left( 1 -  \left(\frac{i}{N} \right)^{r-1} \right)^D \approx 
N \exp \left\{ - D  \left(\frac{i}{N} \right)^{r-1} \right\}. \]
If this is indeed the case we would expect the algorithm to continue until
\[  D  \left( \frac{i}{N} \right)^{r-1}  = \Omega( \log N ). \]
Our main result is that if $ \mc{H}$ satisfies certain 
degree and codegree conditions this is indeed the case.  And in the course of proving this result
we establish a number of other similarities of $I(i)$ and $ S(i) $.

Define the {\em degree} of a set $ A \subset V$ to be the number of edges of 
$ {\mathcal H} $ that contain $A$.
For $ a = 2, \dots, r-1 $ we define $ \Delta_a( {\mathcal H}) $
to be the maximum degree of $A$ over $A \in \binom{V}{a} $.  We also
define the {\em $(r-1)$-codegree} of a pair of distinct vertices $ v,v'$ to be 
the number of edges $e,e' \in {\mathcal H} $ such that $ v \in e \setminus e', v' \in e' \setminus e$ 
and $ |e \cap e'|=r-1 $.  We let $ \G ( \mc{H}) $ be the maximum 
$(r-1)$-codegree of $ {\mathcal H} $.
\begin{theorem}
\label{theory:main}
Let $r\ge 3$ and $ \epsilon > 0 $ be fixed.  Let 
$ {\mathcal H} $ be a $ r$-uniform, $D$-regular hypergraph on $N$ 
vertices such that $ D > N^{\epsilon} $.  If
\begin{equation}
\label{eq:degcond}
 \Delta_\ell( {\mathcal H} ) < D^{ \frac{r-\ell}{r-1} - \epsilon} \ \ \ \text{ for }
\ell = 2, \dots, r-1 
\end{equation}
and $ \Gamma( {\mathcal H} ) < D^{1- \epsilon} $
then the random greedy independent set
algorithm produces an independent set $I$ in $ {\mathcal H} $ with
\begin{equation}
\label{eq:lowerbound}
|I| = \Omega\left( N \cdot \left( \frac{\log N}{ D} \right)^{\frac{1}{r-1}} \right)
\end{equation}
with probability $ 1 - \exp\left\{ - N^{\Omega(1)} \right\} $.
\end{theorem}
\noindent
The proof of Theorem~\ref{theory:main} is given in Section~\ref{sec:lower}.
Similar lower bounds on the independence numbers of $r$-uniform hypergraphs
were established by Ajtai, Koml\'os, Pintz, Spencer and Szemer\'edi \cite{akpss} and
extended by Duke, Lefmann and R\"odl \cite{dlr}.  However, these results are restricted to the
case of simple hypergraphs (i.e. $ \Delta_2( \mathcal{H}) = 1 $).  We note that some condition 
beyond a vertex-degree condition is required to ensure existence of an independent set
of size $ \omega( N/ D^{\frac{1}{r-1}}) $; for example, Cooper and Mubayi \cite{CM} 
recently gave a construction of a
$ K_4^- $-free 3-uniform hypergraph with independence number bounded above by $ 2N/ \Delta^{1/2} $, where
$ \Delta $ denotes maximum degree.

Consider the $H$-free process, where $H$ is a graph with vertex set $V_H$ 
and edge set $ E_H$.  Set $ v_H = | V_H| $ and $ e_H = |E_H| $.  Recall that $H$ is strictly 2-balanced if and only if
\begin{equation}
\label{eq:2balance}  \frac{e_{H[W]} - 1}{ |W| - 2 } < \frac{ e_H - 1}{ v_H -2} \ \ \ \ \ 
\text{ for all } W \subsetneq V_H \text{ such that } |W| \ge 3,
\end{equation}
where $H[W]$ is the subgraph of $H$ induced by $W$.  For $ a \ge 2$ let $ v_a $ be the minimum number of vertices of $H$ spanned by a set of $a$ edges of $H$.  We have
\[  \Delta_a( \mc{H}_H ) = \Theta \left( n^{ v_H - v_a} \right) = \Theta \left( n^{ (v_H-2) \left[ 1 - \frac{ v_a -2}{ v_H -2} \right]} \right) \]
and
\[ D^{ \frac{ e_H-a}{ e_H-1}} = \Theta \left( n^{ \frac{ ( v_H-2)( e_H -a)}{ e_H-1}} \right)
= \Theta \left( n^{ (v_H-2) \left[ 1 - \frac{ a -1}{ e_H -1} \right]}  \right). \]
Thus, we see that $ \mc{H}_H $ satisfies (\ref{eq:degcond}) if and only  $(v_a-2)/(v_H-2) > (a-1)/( e_H-1)$ for all $ a \ge 2$, which
holds if and only if $H$ is strictly 2-balanced.
Thus Theorem~\ref{theory:main} is a generalization of the lower bound on the number of steps in the $H$-free process for $H$ strictly 2-balanced given by Bohman and Keevash \cite{BK}.

We note that the hypotheses of Theorem~\ref{theory:main} imply that $\mc{H}$ is not too dense. 
Specifically, counting $ | \mc{H}|$ in two ways we have
$$ \frac{ND}{r} = | \mc{H} | \le \frac{1}{ \binom{r}{2}} \binom{N}{2} D^{ \frac{r-2}{r-1} - \epsilon}.$$ 
This implies that we have
\begin{equation}
\label{eq:notdense}
N = \Omega \left( D^{\frac{1}{r-1} + \epsilon} \right).
\end{equation}
We make use of this observation in the proofs below.

Some processes for which the degree and codegree conditions 
in Theorem~\ref{theory:main} are relaxed have already been studied.  A {\em diamond} is the graph 
obtained by removing an edge from $K_4$.  The diamond-free process studied
by Picollelli \cite{P1} is an example of an $H$-free process where the graph $H$ is 
2-balanced but not strictly 2-balanced.  When $ H$ is a diamond then the hypergraph $ \mc{H}_H$ is
5-uniform and $ 5(n-2)(n-3)/2 $-regular but has $ 
\Delta_3( \mc{H}_H ) = 3(n-3) = \Theta( D^{1/2} ) $.  For this process Picollelli \cite{P1}
shows that the number of steps is larger than the bound given by (\ref{eq:lowerbound}) 
by a logarithmic factor.  Bennett \cite{pb} has recent results on the sum-free process.
This process is the random greedy independent set algorithm on the hypergraph which has vertex
set $ \mathbb{Z}_n $ and edge set consisting of all solutions of the equations $ a+b=c$.  This
hypergraph does not satisfy the codegree condition in Theorem~\ref{theory:main}. Since $ a+b =c $ 
implies $ (-a) + c =b$, the 2-codegree of $a$ and $-a$ has the same order as the degree $D$ of the 
hypergraph.  Nevertheless, the lower bound (\ref{eq:lowerbound}) still holds for the 
sum-free process.  In both of
these processes, interesting irregularities in $ \mc{H}(i) $ (i.e. violations of our intuition that
$ S(i)$ should resemble $I(i)$)  develop as the process evolves.

%, and these 

%irregularities invalidate the (relatively) simple lower bound proof given here.

It is tempting to speculate that the lower bound in Theorem~\ref{theory:main} 
gives the correct order of magnitude of the maximal independent set produced 
by the random greedy independent set algorithm for a broad class of hypergraphs $ \mc{H}$.  
Bohman and Keevash conjecture that this
is the case for the $H$-free process when $H$ is strictly 2-balanced, but even this remains widely open.  The conjecture has
been verified in some special cases, including the $ K_3$-free process \cite{r3t}, the $K_4$-free process \cite{Wa3,Wz2}
and the $ C_\ell$-free process for all $ \ell \ge 4 $ \cite{P2,P3,Wa2}.

In the interest of communicating a short and versatile proof, we make 
no attempt to optimize (or even explicitly state) the constant in the lower
bound (\ref{eq:lowerbound}).  For the same reason, we also refrain 
from considering hypergraphs that are only approximately $D$-regular.  (It is clear that some
loosening of the regularity condition in Theorem~\ref{theory:main} is possible.)  Our proof uses the so-called differential equations
method for establishing dynamic concentration and is a  
modest simplification of the
earlier work of Bohman and Keevash.  We do not establish self-correcting 
estimates, which are dynamic concentration inequalities with error bounds that improve as the underlying
process evolves.  Such estimates were first deployed by Telcs, Wormald and Zhou \cite{TWZ} (and, independently, in \cite{BP}).  Bohman, Frieze and Lubetzky \cite{bfl} developed a {\em critical interval}
method for proving self-correcting estimates.  Very recently, the critical interval method (and closely related methods) have been used to give a
very detailed analysis of the triangle-removal process \cite{BFL2} and to determine
the asymptotic number of edges in the $K_3$-free process \cite{BK2, FGM}.  The latter gives an 
improvement on Kim's \cite{K} celebrated lower bound
on the Ramsey number $R(3,t)$.

Of course, the cardinality of the maximal independent set produced by the random greedy algorithm is not the
only quantity of interest.  We would also like to understand some of the structural
properties of this set; in particular, what other properties of the binomial random 
set $ S(i)$ are shared by $ I(i)$?  

%One interesting feature of the graph produced

%by the $H$-free process is that it resembles an Erd\H{o}s-R\'enyi random graph $ G_{n,p}$ with the

%same edge density, with the notable exception that $ G_{n,p}$ contains

%many copies of $H$ while the graph produced by the $H$-free process contains no copy of $H$. 

For example, the lower bounds on $ R(s,t) $ mentioned above follow from the fact that the independence number
of the graph produced by the $ K_s$-free process is essentially the same as the independence
number of the corresponding $ G_{n,p}$.  There has been extensive study of the number of copies of
a fixed graph $K$ that does not contain $H$ as a subgraph in the graph produced by the $H$-free 
process \cite{BK, GM, Wa1, Wz3}.  
It turns out that the number of copies of such a graph $K$ is
roughly the same as in the corresponding $ G_{n,p}$.  Our next result is an extension of this
fact to our general hypergraph setting.

Fix $s$ and a $s$-uniform hypergraph $ {\mathcal G} $ on vertex set $V$ (i..e the 
same vertex set as the hypergraph $\mc{H}$).  We let $ X_\mc{G}(i)$ be the number of edges in $\mc{G}$ 
that are contained in the independent set produced at the $i^{th}$ step of the random greedy process on $ \mc{H} $.  Set
$ p = p(i) = i/N$ and let $ i_{\rm max} $ be the lower bound (\ref{eq:lowerbound}) on the size 
of the independent set given by the 
random greedy algorithm given in Theorem~\ref{theory:main}.

\begin{theorem}
\label{theory:count}
If no edge of $ \mc{G}$ contains an edge of $ \mc{H}$, $ i < i_{\rm max}$ is fixed, $ | \mc{G}| p^s \to \infty $ and 
$ \Delta_a( \mc{G}) = o( p^a |\mc{G}|) $ for $ a =1, \dots, s-1$ then
\[ X_\mc{G} (i) = | \mc{G}| p^{s} (1+o(1)). \]
with high probability.
\end{theorem}

\noindent Note that Theorem~\ref{theory:count} does not claim the conclusion holds for all $i$ simultaneously. The proof is given in Section~\ref{sec:count}.

We believe that Theorems~\ref{theory:main}~and~\ref{theory:count} may have multiple applications, 
most notably in the context of the $H$-free process where $H$ is a $k$-uniform 
hypergraph (i.e. $ k \ge 3$ and our vertex set is $ V = \binom{[n]}{k} $).   In Section~\ref{sec:turan} we discuss applications of Theorem~\ref{theory:main} to hypergraph-free processes and note
that in this context Theorem~\ref{theory:main} gives new lower 
bounds on some hypergraph Tur\'an problems.
We also outline one other application: a lower bound on the
number of steps in the $k$-AP-free process.  This process forms a $k$-AP-free subset
of $ \mathbb{Z}_N $ by adding elements chosen uniformly at random one at a time
subject to the condition that no $k$-term arithmetic progression is formed.  Details and
discussion are given in the following Section.

\section{The $k$-AP-free process and Gowers uniformity norm}

\label{sec:gowers}

In this Section we address a question regarding $k$-term arithmetic progressions and the Gowers  
uniformity norm.  The question was asked, independently, by Gowers~\cite{Gowers} and Green~\cite{Green} 
(for further
discussion see Conlon, Fox and Zhao \cite{cfz}).
The Gowers $U^d$ norm of $f : \mb{Z}_N \to \mb{R}$ is 
\begin{equation}
\label{eq:gowers}
\Vert f \Vert_{U^d} := \left[ \frac{1}{N^{d+1}} \displaystyle \sum_{x \in \mathbb{Z}_N , h \in \mathbb{Z}_N^{d}} \;\;\; \prod_{\om\in \{0,1\}^d} f \left(x + h \cdot \om \right) \right]^{1/2^d}.
\end{equation}
Given $ A \subset \mb{Z}_N $ we define a real-valued function 
$ \nu_A = \frac{ N}{|A|} 1_A $.  Motivated by the study of the relationship between
the Gowers norm and the distribution of arithmetic progressions in subsets of $ \mathbb{Z}_N$ (see Section~4 in \cite{Gowers}), Gowers and Green ask
if there exists a function $ s(k)$ such that $ \| \nu_A -1 \|_{U^{s(k)}} = o(1) $ implies that $A$ contains a $k$-term arithmetic progression.

Consider the $k$-AP-free process on $\mathbb{Z}_N$ (the set of integers modulo $N$), where $N$ is prime. This process is an instance of the random greedy independent set algorithm on the $k$-uniform, $k(N-1)$-regular hypergraph $\mc{H}_k$ which has vertex set $ \mb{Z}_N $ and edge set consisting of all $k$-term arithmetic progressions.  We apply Theorems~\ref{theory:main}~and~\ref{theory:count} to prove the following.

\begin{corollary}

\label{cor:jacob}

Let $k \ge 3$ and $d$ be fixed integers such that $ 2^{d-1} = k-1 $.  Let $N$ be prime.
With high probability the $k$-AP-free process produces a set $I(i_{\rm max}) \subseteq \mathbb{Z}_N $ such that
\[ \| \nu_I -1 \|_{U^{d}} =o(1). \]
\end{corollary}

\noindent Of course, the set $I$ contains no $k$-term arithmetic progression.  So we conclude that if 
the function $s(k)$ exists then it satisfies $s(k) > 1+ \log_2 (k -1)$.  The remainder of this Section is a proof of Corollary~\ref{cor:jacob}.

We begin by noting that $ \mc{H}_k $ satisfies the conditions required for an application of 
Theorem~\ref{theory:main}.  The condition on $ \Delta_a $ follows from the fact that
any $2$ elements of $ \mathbb{Z}_N $ are in at most $k^2$ edges.   Furthermore, 
for any $v,v' \in \mathbb{Z}_N$ there are at most $k^6$ pairs of edges $e,e'$ such that
$v \in e, v' \in e'$, and $|e \cap e'| \ge 2$.  (Observe that setting the positions of 2 vertices in $ e \cap e'$ and $v,v'$ in the two arithmetic progressions  $e$ and $e'$ introduces a pair of linear equations that the differences for $e$ and $e'$ satisfy.  This determines these differences uniquely.)  Thus, we also have the desired condition on $ \Gamma $.  We conclude that with high probability the $k$-AP-free process produces a $k$-AP free set $I$ of size $$\Om \left(N^\frac{k-2}{k-1} \log^\frac{1}{k-1} N \right).$$

In order to compute $ \| \nu_I - 1 \|_{U^d} $ we consider the hypergraph 
$ \mc{G} $ of `$d$-cubes.'  For $ x \in \mathbb{Z}_N $ and $ h \in \mathbb{Z}_N^d $ define
$$ e_{x,h} = \left\{x + \om \cdot h: \om \in \{0,1\}^d  \right\}. $$
We say that $ h \in \mathbb{Z}_N^d $ has {\em no coincidences} if the elements of
$ \left\{\om \cdot h: \om \in \{-1,0,1\}^d  \right\} $ are distinct.  In other words, $h$ 
has no coincidences if
\[  \left|\left\{\om \cdot h: \om \in \{-1,0,1\}^d  \right\} \right| = 3^d. \]
We then define $ \mc{G} $ to be the hypergraph of $d$-cubes that have no coincidences; that is, we set
$$ \mc{G}  = \left\{  e_{x,h} : h \text{ has no coincidences} \right\}. $$
We now establish the conditions required for an application of
Theorem~\ref{theory:count} to the number of edges of $ \mc{G}$ that appear in $ I$.  Note that $ | \mc{G} | = (1+o(1)) N^{d+1} $.
As $p = |I|/N = \tilde{ \Theta} ( N^{- \frac{1}{k-1}} ) $, we have
\[ p^{2^d} | \mc{G}| =  \tilde{\Theta}( N^{-2} N^{d+1} ) \to \infty .\]
Since we do not allow $d$-cubes with coincidences, no element of $ {\mathcal G} $ contains a $k$-AP  (If $ h \cdot y - h \cdot x = h \cdot z - h \cdot y $ then $ h \cdot (y-x) $ coincides with $ h \cdot (z-y) $.)   The degree condition needed for an application of 
Theorem~\ref{theory:count} is an easy consequence of the following Lemma.

\begin{lemma}
\label{lem:extender}
For $1 \le a \le 2^d$, any set of $a$ vertices is contained in at most $O \left(N^{d-\lceil\log_2 a \rceil} \right)$ edges of $\mc{G}$.
 \end{lemma}

 \begin{proof}
First, note that there are $\left(2^d \right)_a$ ways to specify which element of $ \{ 0,1\}^d $ corresponds to 
each vertex in our set of $a$ given vertices.  So, it suffices to prove the following: 
For any $ y_1, \dots, y_a \in {\mathbb Z}_N $, not necessarily distinct, and $ \om_1, \dots, \om_a \in \{0,1\}^d$, 
distinct, there are at most $O \left(N^{d-\lceil\log_2 a \rceil} \right)$ choices of $ x \in {\mathbb Z}_N$ and
$ h \in {\mathbb Z}_N^d $ such that $ y_j= x + \om_j \cdot h $ for $j=1, \dots, a$.

We prove this statement by induction on $d$.  Note that this statement is trivial if $a=1$.  This dispatches
the base case ($d=1$) and allows us to hence forth consider $ a \ge 2$.
%
%Suppose $a, d\ge 2$ and  we are given $y_1 \ldots y_a \in \mathbb{Z}_N$.  
%We set $y_j = x + \om_j \cdot h$ for $1 \le j \le a$. %WLOG we assume that it is not possible to deduce the value of $x+\vec{\om} \cdot %\vec{h}$ for any other vectors $\vec{\om}$ from the information we have so far. 

WLOG assume that the Hamming distance between $\om_1$ and $ \om_2$ is minimal among distances between pairs of vectors from $\om_1, \ldots, \om_a$, and suppose that $L \subset \{1, \ldots, d \}$ is the set of coordinates at which $\om_1$ and $\om_2$ differ. Then there are $O\left( N^{|L|-1} \right)$ ways to specify the coordinates $h_\l$ for $\l \in L$ which are consistent with $y_1= x + \om_1 \cdot h$ and $ y_2 = x + \om_2 \cdot h$.

Now, consider one of the ways to set the coordinates $L$.  We may view the remainder of the embedding as a lower dimensional cube, with $d' = d - |L|$ and $a' \ge \frac{1}{2} a$ (since if there were three vectors $\om_j, \om_{j'}, \om_{j''}$ that only differed in coordinates of $L$, then two of them would have Hamming distance less than $|L|$, contradicting the fact that the Hamming distance from $\om_1$ to $\om_2$ is minimal).

Appealing to the induction hypothesis, altogether there are $$O\left(N^{|L|-1} \cdot N^{(d-|L|) - \left\lceil\log_2 \left(\frac{1}{2} a \right) \right\rceil} \right) = O \left(N^{d- \lceil\log_2 a \rceil} \right)$$ possible $x, h$.

 \end{proof}

We are now prepared to compute the uniformity norm for $ I =I(i_{\rm max})$.
Recall that $p= \frac{|I|}{N}$. Note that 
$\nu_I = \frac{1}{p} 1_I$.  We calculate $\Vert \nu_I - 1 \Vert_{U^d}$ by first considering the $h \in {\mathbb Z}_N^d $ that have no coincidences.  We will see below that the contribution from $ h $ with the property that $ \{ h \cdot \om : \om \in \{-1,0,1\}^d \} $ are not all distinct is negligible.

Consider the number of edges of the hypergraph $\mc{G}$ that are contained in $I$.  As we note above, we can apply Theorem~\ref{theory:count} to get an estimate for this number. Now note that if the hypotheses of Theorem~\ref{theory:count} hold for $\mc{G}$ then they also hold for the shadow of $\mc{G}$ on the $x^{th}$ level (i.e. the hypergraph on vertex set $V$ whose edges are all sets of size $x$ contained in an edge of $\mc{G}$).   We conclude that for each $0 \le x \le 2^d$, w.h.p. the number of edges of $\mc{G}$ with exactly $x$ vertices in $I$ is  $$(1+o(1)) N^{d+1} \binom{2^d}{x} p^x.$$ Thus to compute $\Vert \nu_I - 1 \Vert_{U^d} = \Vert \frac{1}{p} 1_I - 1 \Vert_{U^d}$ using (\ref{eq:gowers}), we start with the sum of the terms corresponding to edges of $\mc{G}$. We get that 

\begin{equation*}
\begin{split}
\sum_{e \in \mc{G}} \left( \frac{1}{p} -1 \right)^{ |e \cap I|} (-1)^{2^d - |e \cap I|} 
& =
N^{d+1} \displaystyle \sum_{0 \le x \le 2^d} (1+o(1)) \binom{2^d}{x} p^x \left(\frac{1}{p} - 1\right)^x (-1)^{2^d - x} \\ 
& = N^{d+1}[(1-p) - 1]^{2^d} + o\left(N^{d+1} \right) \\
& = o\left(N^{d+1} \right)
\end{split}
\end{equation*}
with high probability.

It remains to address the terms in (\ref{eq:gowers}) corresponding to $h$ such that the values $ \{ h \cdot \om: \om \in \{-1,0,1 \}^d \} $ are not all distinct. Each such vector $h$ defines a partition $\mc{P}$ of $\{-1,0,1\}^d$: each part $P \in \mc{P}$ is a maximal subset of $ \{-1,0,1\}^d$ with the property that the values  $ h \cdot \om$ are the same for all $\om \in P$.

We compute the remaining contribution to (\ref{eq:gowers}) by summing over all possible partitions.  All vectors $h$ that define a given partition $\mc{P}$ satisfy a system of linear equations: namely, we have $h \cdot \om = h \cdot {\om}' $ for every pair $\om, {\om}'$ in the same part of $\mc{P}$. Suppose these equations give the matrix equation $A h = 0$, and assume $A$ has rank $a$. Then there are $O\left(N^{d-a+1} \right)$ pairs $x, h$ that respect this partition $\mc{P}$.

The partition $ \mc{P} $ of $ \{-1,0,1\}^d $ defines a partition $ \overline{ \mc{P}} $ of $ \{0,1\}^d $ by restriction.
(Of course, we only interested in $ \om \in \{0,1\}^d$ when computing the contribution to the Gowers norm).  We claim that each part $P$ of $\overline{\mc{P}}$ we have $ | P | \le 2^a$. To see this, let $\om_0 \in P$ where $ P \in \overline{\mc{P}}$ and note that for every $\om \in P$, the rowspace of $A$ contains $\om- \om_0$. But a subspace of $\mathbb{Z}_N^d$ of dimension $a$ can only intersect $ \{0,-1\}^y \times \{0,1\}^{d-y}$ in at most $2^a$ points.  So the rowspace of $A$ can only contain $2^a$ vectors that are $0$ or $-1$ on the support of $\om_0$ and $0$ or $1$ otherwise. Thus, $|P| \le 2^a$.

For the partition $\mc{P}$, there are $O\left( N^{d-a+1} \right)$ pairs $x, h$ that agree with $ \mc{P}$.  Fix such a pair $x, h$, and a collection of parts $\mc{S} \subset \overline{\mc{P}}$. Consider the event $\mc{E}_{\mc{S}}$ that the images (under the map $ \varphi: \{0,1\}^d \to \mathbb{Z}_N$ defined by $ \varphi( \om) = x + h \cdot \om$) of the parts of $\mc{S}$ are in $I$, but none of the image of parts of $\overline{\mc{P}} \setminus \mc{S}$ are in $I$. By a simple first moment calculation (using Lemma~\ref{lem:wolf}), we have $$ \mathbb{P} \left[\mc{E}_\mc{S} \right] = O\left(p^{|\mc{S}|}\right).$$
The sum of the terms in $\Vert \nu_I - 1 \Vert_{U^d}$ corresponding to pairs $x, h$ such that $h$ respects the partition $\mc{P}$ is 
\begin{multline*}
O\left( N^{d-a+1} \displaystyle \sum_{\mc{S} \subset \overline{\mc{P}}} p^{|\mc{S}|} \left(\frac{1}{p} - 1\right)^{|\cup_{P \in \mc{S}} P| }   \right) \\ = O\left( N^{d-a+1}  p^{|\overline{\mc{P}}|- 2^d} \right) = O\left( N^{d-a+1} \cdot p^{2^{d-a}-2^d} \right) = o\left( N^{d+1} \right),
\end{multline*}
where we use $ k-1 = 2^{d-1} $ in the last equation.  As there are only finitely many partitions $ \mc{P} $, the proof of Corollary~\ref{cor:jacob} is complete.

\section{Hypergraph Tur\'an}

\label{sec:turan}

Consider the $H$-free process, where $H$ is a $k$-uniform hypergraph ($k\ge 2$) with vertex set $V_H$ 
and edge set $ E_H$.  Set $ v_H = | V_H| $ and $ e_H = |E_H| $.  Recall that the $H$-free process
is the same as random greedy independent set on the hypergraph $ {\mathcal H}_H $ that has vertex 
set $ \binom{[n]}{k} $ and edge set consisting of all copies of $H$.  
We say that $H$ is strictly $k$-balanced if and only if
\begin{equation}
\label{eq:2balanceH}  \frac{e_{H[W]} - 1}{ |W| - k } < \frac{ e_H - 1}{ v_H -k} \ \ \ \ \ 
\text{ for all } W \subsetneq V_H \text{ such that } |W| > k,
\end{equation}
where $H[W]$ is the subgraph of $H$ induced by $W$.  We claim that $ \mc{H}_H $ satisfies (\ref{eq:degcond}) if and only if $H$ is strictly $k$-balanced. To see this,  note that in this setting we have $r=e_H$, $N = \Theta\of{n^k}, D=\Theta\of{n^{v_H-k}}$.  For 
$ a \ge 2$ let $ v_a $ be the minimum number of vertices of $H$ spanned by a set of $a$ edges of $H$.  We have
\[  \Delta_a( \mc{H}_H ) = \Theta \left( n^{ v_H - v_a} \right) = \Theta \left( n^{ (v_H-k) \left[ 1 - \frac{ v_a -k}{ v_H -k} \right]} \right) \]
and
\[ D^{ \frac{ e_H-a}{ e_H-1}} = \Theta \left( n^{ \frac{ ( v_H-k)( e_H -a)}{ e_H-1}} \right)
= \Theta \left( n^{ (v_H-k) \left[ 1 - \frac{ a -1}{ e_H -1} \right]}  \right). \]
Thus, we see that $ \mc{H}_H $ satisfies (\ref{eq:degcond}) if and only  $(v_a-k)/(v_H-k) > (a-1)/( e_H-1)$ for all $ a \ge 2$, which
holds if and only if $H$ is strictly $k$-balanced.  Furthermore, if $H$ has no vertex of degree $1$ then $\Gamma(\mc{H}_H) = O(n^{v_H - k -1})$.

Therefore, Theorem~\ref{theory:main} 
implies lower bounds on the Tur\'an numbers of $k$-uniform, strictly $k$-balanced hypergraphs.
\begin{corollary}If $H$ is a $k$-uniform, strictly $k$-balanced hypergraph with $v_H$ vertices, $e_H \ge 3$ edges, and no vertex of degree $1$ then $$ex(n, H) = \Omega\left( n ^ { k  - \frac{ v_H - k}{e_H - 1}}    \log ^ \frac{1}{e_H-1} n\right).$$
\end{corollary}
\noindent
For general $k$-partite, strictly $k$-balanced hypergraphs, the bounds we give are best known. 
(Although for a few fixed hypergraphs better bounds are known; for example, 
see \cite{t1},\cite{t2}.) Note that the complete 
$k$-partite $k$-uniform hypergraph $K_{s_1, \ldots, s_k}$ is strictly $k$-balanced so long 
as all $s_i \ge 1$ and $s_i, s_i' \ge 2$ for some $i\neq i'$. To see this, first note that the condition $s_i, s_i' \ge 2$ implies there are no vertices of degree $1$. Now suppose $W$ is a set of vertices that 
has $w_i$ from the $i^{th}$ part. Then condition \eqref{eq:2balanceH} becomes  
\[\frac{\prod w_i - 1}{ \sum w_i - k } < \frac{ \prod s_i - 1}{ \sum s_i -k}\]
for all sequence of integers $ w_1, w_2, \dots, w_k $ such that $ 1 \le w_i \le s_i $ for all 
$i$, $w_i > 1$ for some $i$, and $ w_i < s_i $ for some $i$.  Straightforward multivariate 
calculus shows that the above inequality holds for all real numbers $w_1, \ldots, w_k$ such that 
$1 \le w_i \le s_i$, except for the case when $w_i=1$ for all $i$, and the case when $w_i=s_i$ for all $i$.

\section{Lower bound: Proof of Theorem~\ref{theory:main}}

\label{sec:lower}

We use dynamic concentration inequalities
to prove that carefully selected statistics remain very close to their expected trajectories throughout the
process with high probability.  Our main goal is to prove dynamic concentration of $|V(i)|$, which is the number of vertices that 
remain in the hypergraph.  In order to achieve this goal, we also track the following variables:
For every vertex $ v \in V(i)$ and $ \ell =2, \dots, r$ define
$ d_\ell(i,v) = d_{\ell}(v)$ to be the number of edges of cardinality $ \ell$ in $ \mc{H}(i) $ that contain
$v$.

We employ the following conventions throughout this section. 
If we arrive at a hypergraph $ \mc{H}(i) $ that has edges $e,e'$ such that $ e \subseteq e'$ then we
remove $ e'$ from $ \mc{H}$.  Note that this has no impact in the process as the presence of $ e$ ensures that
we never have $ e' \subset I(j)$.  For any variable $X$ we use the notation $ \D X := X(i+1) - X(i) $ for 
the one step change in $X$.  Since {\em every} expectation taken in this section is conditional on the first $i$ steps
of the algorithm, we suppress the conditioning.  That is, we simply write $ \E[ \, \cdot \, ] $ instead 
of $ \E[ \, \cdot \mid \mc{F}_i] $ where $ \mc{F}_0, \mc{F}_1, \dots $ is the natural filtration generated by the algorithm.  If $ f$ and $g$ are functions of $N$ with the property that $ f$ is bounded by $g$ times some poly-logarithmic function (in $N$) then we write $ f= \tilde{O}(g) $.

We begin by discussing the expected trajectories of the variables we track.  Here we use the binomial random set $ S(i) $ as a guide.  Recall that each vertex is in $S(i)$, independently, with probability $ p = p(i)= i/N$.  Let $v$ be a fixed vertex.  
The expected number of edges $e \in \mc{H}$ such that $ v \in e $ and $ e \setminus \{v\} \subseteq S(i)$ is
$$D p^{r-1} = D \left(\frac{i}{N} \right)^{r-1} = t^{r-1}$$
where we parametrize time by setting $t:=\frac{D^{\frac{1}{r-1}}}{N} \cdot i$.  Thus, we set 
$$q = q(t):=e^{-t^{r-1}}$$ and think of $q$ as the probability that a vertex is in $V(i)$.  So, we should
have  $|V(i)| \approx q(t) N$ and $d_\l(v)$ should follow the trajectory 
$$ s_\l(t) := D \binom{r-1}{ \l -1} q^{\l-1} p^{r-\l} = \binom{r-1}{ \l-1}D^{\frac{\l-1}{r-1}}t^{r-\l}q^{\l-1}.$$
For the purpose of our analysis, we separate the positive contributions to $d_\l(v)$ from the negative contributions. 
We write $ d_r(v) = D - d_r^-(v) $, and for $ \ell < r $ we write
$ d_\l(v) = d_\l^+(v) - d_\l^-(v)$, where $d_\l^+(v), d_\l^-(v)$ are non-negative variables which count the number of edges of cardinality $\ell$ containing $v$ that are created and destroyed, respectively, through the first $i$ steps of the process.  We define
$$s_\l ^+(t):= D^{-\frac{1}{r-1}} \int_0^t \frac{\l s_{\l+1}(\t)}{q(\t)} d \t \;\;\;\;\;\;\;\;\;\;\; s_\l ^-(t):= D^{-\frac{1}{r-1}} \int_0^t \frac{(\l-1) s_{\l}(\t) s_2(\t)}{q(\t)} d \t,$$
and claim that should have $ d_\l^{\pm} \approx s_\l^{\pm} $.  
Note that $s_\l$ satisfies the differential equation 
\[s_\l ' =\frac{\l s_{\l+1} -(\l-1) s_{\l} s_2 }{q} = (s_\l^+) ' - (s_\l^-) '\] and so $s_\l = s_\l^+-s_\l^-$. The choice of $s_\l^\pm$ is natural in light of the usual
mechanism for establishing dynamic concentration and the observation that we have
$$ \E [\D d_\l^+ ] \approx \frac{1}{|V(i)|} \cdot \l d_{\l+1} \;\;\;\;\;\;\;\;\;\;  \E[ \D d_\l^- ] \approx \frac{1}{|V(i)|} \cdot (\l-1) d_{\l}d_2.$$
%(In the above approximations and in future calculations we suppress the index $i$ and omit delimiters ``$|\cdot|$" signifying set cardinality, so for example instead of $|V(i)|$ we just write $V$ for the number of vertices at step $i$.)

%\section{Stopping time}\label{sec:stoppingtime}

In addition to our dynamic concentration estimates, we need some auxiliary information about the evolving hypergraph $ \mc{H}(i)$.  
\begin{definition}[Degrees of Sets]
For a set of vertices $A$ of at least 2 vertices, let $d_{A \up b}(i)$ be the number of edges of size $b$ containing $A$ in 
$ \mc{H}(i) $.
\end{definition}

\begin{definition}[Codegrees]
For a pair of vertices $v,v'$, let $c_{a, a' \to k}(v,v',i)$ be the number of pairs of edges $e,e'$, such that 
$v \in e \setminus e', v' \in e' \setminus e$, $|e|=a, |e'|=a'$ and $|e \cap e'|=k$ and $ e, e' \in \mc{H}(i)$.
\end{definition}

\noindent
We do not establish dynamic concentration for these variables, but we only need relatively crude upper 
bounds.

In order to state our results precisely, we introduce a stopping time.  Set
$$i_{\rm max} := \z N D^{-\frac{1}{r-1}} \log^\frac{1}{r-1} N
\ \ \ \ \ \ \ \ \text{ and }  \ \ \ \ \ \ \ \ \ t_{\rm max} := \frac{D^{\frac{1}{r-1}}}{N} \cdot i_{\rm max} = \z  \log^\frac{1}{r-1} N,$$
where $\z>0$ is a constant (the choice of this constant is discussed below). Note that we have
\begin{equation}
\label{eq:minq}
q(t_{max}) = N^{-\zeta^{r-1}}.
\end{equation}
Define the stopping time $T$ as the minimum of $ i_{\rm max}$ and the first step $i$ such that 
any of the following conditions fails to hold:
\begin{align} 
\label{eq:points}
| V(i) | & \in Nq  \pm N D^{-\d} f_v \\
\label{eq:vertexdegree}
d_\l^\pm(v)  & \in s_\l^\pm  \pm D^{\frac{\l-1}{r-1}-\d}f_\l \ \ \ \ \text{ for } \ell = 2, \dots, r \text{ and all } v \in V(i) \\ 
%d_\l^-(v) & \in s_\l^-  \pm D^{\frac{\l-1}{r-1}-\d}f_\l \ \ \ \ \text{ for } \ell = 2, \dots, r \text{ and all } v \in V(i) \\
\label{eq:setdegree}
d_{A \up b} & \le D_{a \up b}   \ \ \ \ 
\text{ for } 2 \le a < b \le r \text{ and all }   A \in \binom{V(i)}{a} \\
\label{eq:codegree}
c_{a, a' \to k}(v, v')  &\le  C_{a,a' \to k}
\ \ \ \ \text{ for all } v,v' \in V(i)  
\end{align}
where $ \d >0 $ is a constant and $ f_v, f_2, \dots, f_r $ are functions of $t$ (which will stay small enough so that the error terms are always little-o of the main terms) and 
$D_{a \up b}$ and $C_{a,a' \to k}$ are functions of $D$ (but not $t$) that satisfy
\begin{align*}
D_{a \up b} & \le  D^{ \frac{ b-a}{r-1} - \frac{ \e}{2}} \\
C_{a,a' \to k}  & \le D^{ \frac{ a + a' -k -2}{r-1} - \frac{ \e}{2}}.
\end{align*}
All of these parameters are specified explicitly below.

We prove Theorem~\ref{theory:main} by showing that $ \P (T < i_{\rm max}) < \exp \left\{ - N^{\Omega(1)} \right\} $.   We break the proof into two parts.  We first establish the crude bounds, namely (\ref{eq:setdegree}) and (\ref{eq:codegree}), in Section~\ref{sec:crude}.  We then turn to the dynamic concentration inequalities (\ref{eq:points}) and (\ref{eq:vertexdegree}) in Section~\ref{sec:dynamic}.

The constants $ \z, \d$ 
are chosen so that
\[ \z \ll \d \ll \e, \]
in the sense that $ \d $ is chosen to be sufficiently small with respect to $ \e$, and $ \z$ is chosen to be sufficiently small with respect to $ \d$.  The martingales that we consider below are stopped in the sense that when we define a sequence
$ Z(i)$ we in fact work with $ Z(i \wedge T)$.  Thus we can assume that the bounds (\ref{eq:points})- (\ref{eq:codegree}) always hold.  The martingales that depend on a fixed vertex $v$ (or a fixed sets of vertices $A$) are also {\em frozen} in the sense that we set $ Z(i) = Z(i-1) $ if the vertex $v$ (or some vertex in the fixed set $A$) is not in $ V(i)$.

%The main theorem will follow once we prove that (for appropriate choices of $D_{a \up b}$, $C_{a, a', k}$, $f_v$, $f_\l$, and $\d$), with high probability the stopping time $T = i_{end}$.

%

%Note that before $T$ we have $Q \ge \frac{1}{2} n^{1- \z^{r-1}} \ge \frac{1}{2} n D^{- \lambda}$, where $\lambda := \frac{\z^{r-1}}{\e}$  (pending our choice of a function $f_v$). 

\subsection{Crude bounds}

\label{sec:crude}

%The bounds we prove here are actually slightly stronger than required to maintain conditions (\ref{eq:setdegree}) and (\ref{eq:codegree}).  

Define
\begin{align*}
D_{a \up b} & := D^{\frac{b-a}{r-1}-\e + 2(r-b)\lambda}  \\
C_{a, a' \to k}  &:= 2^r D^{\frac{a+a'-k-2}{r-1}-\e+ (2r-2k-2)\lambda}
\end{align*}
where $ \lambda = \e/4r $.  Throughout this section we use the bound
$ |V(i)| > N D^{-\lambda}$, which follows from (\ref{eq:minq}) and the fact that we may set 
$ \zeta >0 $ sufficiently small relative to $\epsilon$.
\begin{lemma}\label{dlemma}
Let $ 2 \le a < b \le r $.
\[ \P \left( \exists i \le T \text{ and } A \in \binom{ V(i)}{a} \text{ such that }   d_{A \up b} (i) \ge D_{a \up b} \right) \le \exp \left\{- N^{ \Omega(1)} \right\}.\] 
\end{lemma}

\begin{proof}
We go by reverse induction on $ b $.  Note that if $b=r$ then the desired bound follows immediately from
the condition on $\D_a( \mc{H} )$ assumed in the statement of Theorem~\ref{theory:main}.

Let $b < r$ and consider a fixed $ A \in \binom{V}{a} $.  For $0 \le j \le D_{a+1 \up b+1}$, let $N_j(i)$ be the number of vertices in $V(i)$ (but not in $A$) that appear in $j$ of the edges counted by $d_{A \up b+1}(i)$. Note that clearly $\sum N_j(i) = |V(i)|$,  and $\sum jN_j(i) = (b+1-a) d_{A \up b+1} \le (b+1-a) D_{a \up b+1}$ .  %Let $m = m(i)$ be the

%smallest number such that

%$$ \sum_{k \ge m} N_k(i) \le \frac{1}{2} N D^{-\lambda} .$$

Then (w.r.t.\ the filtration $\mc{F}_i$) $\Delta d_{A \up b}( i)$ is stochastically dominated by $\Delta X(i)$, where $X(i)$ is a variable such that $X(0)=0$ and 
$$ \P(\D X = j) = \frac{N_j(i)}{ |V(i)| } $$
%\frac{1}{2} n D^{-\lambda}}, \;\;\;\;\;\; k(i) \le j \le D_{a+1 \up b+1}.$$
We will bound $d_{A \up b}( i)$ by appealing to this stepwise stochastic domination, and bound $X(i)$ using the following lemma due to Freedman \cite{F}:

\begin{lemma}[Freedman] \label{lem:Freedman}
Let $Y(i)$ be a supermartingale, with $\Delta Y(i) \leq C$ for all $i$, and $V(i) :=\displaystyle \sum_{k \le i} Var[ \Delta Y(k)| \mathcal{F}_{k}]$  Then
$$ \P\left[\exists i: V(i) \le v, Y(i) - Y(0) \geq d \right] \leq \displaystyle \exp\left(-\frac{d^2}{2(v+Cd) }\right).$$ \end{lemma}
\noindent
To apply the lemma, we calculate
$$ \E[\D X] = \frac{1}{|V(i)|} \sum j N_j \le \frac{1}{ND^{-\lambda}} (b+1-a) D_{a \up b+1}  \le \frac{r}{N} D^{\frac{b-a+1}{r-1} - \e + \left(2r-2b-1 \right) \lambda}.$$
Thus if we define $$Y(i):=  X(i) - \frac{r}{N} D^{\frac{b-a+1}{r-1} - \e + \left(2r-2b-1 \right) \lambda}\cdot i$$ then $Y(i)$ is a supermartingale. Now
%so for all $i \le T$ we have $$E[X(i)]   \le 2D^{\frac{b - a}{r-1} + \left(r-b-\frac{1}{2} \right) \lambda }$$
\begin{multline*}
Var[\D Y] = Var[\D X] \le \E \left[(\D X)^2 \right] = \frac{1}{ |V(i)| } \sum{j^2 N_j(i)} \\
\le \frac{D_{a+1 \up b+1}}{ |V(i)|} \sum{j N_j} \le \frac{D_{a+1 \up b+1}}{ |V(i)|} \cdot r D_{a \up b+1} \le \frac{r}{N} D^{\frac{2b-2a+1}{r-1} -2\e + \left(4r-4b-3 \right) \lambda}
\end{multline*}
So we apply Lemma~\ref{lem:Freedman} with $$v =  (\log N )D^{\frac{2b-2a}{r-1} - 2\e + \left(4r-4b-3 \right) \lambda},$$
recalling (\ref{eq:notdense}),
 and $C= D_{a+1 \up b+1} = D^{\frac{b-a}{r-1} - \e + (2r-2b-2) \lambda}$ to conclude that we have
%$d = D^{\frac{b-a}{r-1} - \e + \left(2r-2b - 1\right) \lambda}$. 
%
\[ P \left[Y(i) \ge D^{\frac{b-a}{r-1} - \e + \left(2r-2b - 1\right) \lambda} \right]
\le \exp\left\{ -N^{\Om(1)} \right\} .\]
%\begin{eqnarray*}
%&&P \left[Y(i) \ge D^{\frac{b-a}{r-1} - \e + \left(2r-2b - 1\right) \lambda} \right] \\
%&&\le \exp \left\{ - \frac{\left(D^{\frac{b-a}{r-1} - \e + \left(2r-2b - 1\right) \lambda} \right)^2}{2\left(2D^{\frac{2b-2a}{r-1} - 2\e + \left(4r-4b-3 \right) \lambda}+D^{\frac{b-a}{r-1} - \e + (2r-2b-2) \lambda} \cdot D^{\frac{b-a}{r-1} - \e + \left(2r-2b - 1\right) \lambda}  \right)} \right\}\\
%&& \le \exp\left\{ -n^{\Om(1)} \right\}
%\end{eqnarray*}
%
%Thus w.h.p. we have $X(T) \le Y(T) + \frac{3}{n} D^{\frac{b-a+1}{r-1} - \e + \left(2r-2b-1 \right) \lambda}\cdot i_{end} < D_{a \up b}$
This suffices to complete the proof (applying the union bound over all choices of the set $A$).
\end{proof}

\begin{lemma}\label{clemma}

Let $ 2 \le a, a' \le r $ and $ 1 \le k <a,a' $ be fixed.
\[ \P \left( \exists i \le T \text{ and } v, v' \in V(i) \text{ such that }   c_{a,a' \to k}(v,v',i) \ge
C_{a, a' \to k} \right) \le \exp \left\{ -N^{\Omega(1)} \right\}. \]
\end{lemma}

\begin{proof} We first note that Lemma~\ref{dlemma} implies Lemma~\ref{clemma} except in the case $a=a'=k+1$. To see this, suppose $a' > k+1$ and let $v, v'$ be any two vertices. Then we have
$$c_{a,a' \to k}(v,v') \le d_a(v) \cdot \binom{a}{k} \cdot D_{k+1 \uparrow a'} \le D^{ \frac{a-1}{r-1}} \cdot 2^r \cdot D^{ \frac{a' - (k+1)}{r-1} - \e + 2(r-a')\lambda},$$ 
which gives the desired bound. 

 So we restrict our attention to the case $a=a'=k+1$.  We again proceed by induction, with the base case following immediately from the condition on $ \Gamma( \mc{H} )$. Note that $c_{k+1, k+1 \to k}(v, v',i)$ can increase in size only when the algorithm chooses a vertex contained in the intersection of a pair of edges from $c_{k+2, k+2 \to k+1}(v, v',i)$, or when the algorithm chooses the vertex not contained in the intersection of a pair of edges from $c_{k+2, k+1 \to k}(v, v',i)$ or $c_{k+1, k+2 \to k}(v, v',i)$. Also, on steps when $c_{k+1, k+1 \to k}(v, v',i)$ does increase, it increases by at most $2D_{2 \up k+2} + D_{2 \up k+1} \le 3D^{\frac{k}{r-1} - \e + (2r-2k-4) \lambda}$.

For $0 \le j \le 3D^{\frac{k}{r-1} - \e +(2r-2k-4) \lambda}$, let $N_j(i)$ be the number of vertices that, if chosen, would increase $c_{k+1, k+1 \to k}(v, v')$ by $j$. Note that $\sum N_j(i) = |V(i)|$ while 
$$\sum jN_j(i) \le C_{k+2, k+2 \to k+1}+ C_{k+2, k+1 \to k} + C_{k+1, k+2 \to k} \le 3 \cdot 2^r D^{\frac{k+1}{r-1}-\e + (2r-2k-2) \lambda}.$$
%
%Let $k(i):= \min \left\{ k: \displaystyle\sum_{j \ge k} N_j(i) \le \frac{1}{2} n D^{-\lambda} \right\}$
%
Then (w.r.t.\ the filtration $\mc{F}_i$)   $\Delta c_{k+1, k+1 \to k}(v, v')(i)$ is stochastically dominated by $\Delta X(i)$, where $X(i)$ is a variable such that $X(0)=0$ and 
$Pr(\D X = j) = \frac{N_j(i)}{ |V(i)| } $.

We apply Lemma~\ref{lem:Freedman} to bound $X$. 
We define $$Y(i):=  X(i) - \frac{3 \cdot 2^r}{N} D^{\frac{k+1}{r-1}-\e + (2r-2k-1) \lambda}\cdot i$$ and note that 
$Y(i)$ is a supermartingale. Now
\begin{multline*}
Var[\D Y] = Var[\D X] \le \E \left[(\D X)^2 \right] = \frac{1}{ |V(i)|} \sum{j^2 N_j(i)} \\
\le \frac{3D^{\frac{k}{r-1} - \e +(2r-2k-4) \lambda}}{N D^{- \lambda}} \sum{j N_j} 
%\le \frac{3D^{\frac{k}{r-1} - \e +(2r-2k-4) \lambda}}{\frac{1}{2} n D^{- \lambda}} \cdot 3 \cdot 2^r D^{\frac{k+1}{r-1}-\e + (2r-2k-2) \lambda}\\
%&& 
\le \frac{9 \cdot 2^r}{N} D^{\frac{2k+1}{r-1} - 2\e + (4r-4k-5) \lambda}
\end{multline*}

Applying Lemma~\ref{lem:Freedman} with $$v = ( \log N) D^{\frac{2k}{r-1} - 2\e + (4r-4k-5) \lambda},$$
recalling (\ref{eq:notdense}),
 and $C= 3D^{\frac{k}{r-1} - \e + (2r-2k-4) \lambda}$ we have 
\[ P \left[Y(i) \ge D^{\frac{k}{r-1} - \e + \left(2r-2k-2.1 \right) \lambda} \right] 
%&&\le \exp \left\{ - \frac{\left(D^{\frac{k}{r-1} - \e + \left(2r-2k-2.1 \right) \lambda}\right)^2}{2\left(18 \cdot 2^r D^{\frac{2k}{r-1} - 2\e + (4r-4k-5) \lambda}+3D^{\frac{k}{r-1} - \e+(2r-2k-4) \lambda} \cdot D^{\frac{k}{r-1} - \e + \left(2r-2k-2.1 \right) \lambda} \right)} \right\}\\
\le \exp\left\{ -N^{\Om(1)} \right\}. \]

%

%Thus w.h.p. we have $X(T) \le Y(T) + \frac{7 \cdot 2^r}{n} D^{\frac{k+1}{r-1}-\e + (2r-2k-2) \lambda}\cdot i_{end} < C_{k+1, k+1, k}$

%

%

\end{proof}

%Note that if we let $$\k:= \e - 2r \lambda$$ then we have the nicer looking bounds

%$$D_{a \up b} \le D^{\frac{b-a}{r-1} - \k} \;\;\;\;\;\;\;\;\;\ C_{a, a', k} \le D^{\frac{a+a'-k-2}{r-1} - \k}.$$

\subsection{Dynamic concentration}

\label{sec:dynamic}

Consider the sequences
\begin{align*}
Z_{V} & := |V(i)| - Nq - ND^{-\d}f_v \\
Z_{\l}^+(v) & := d_\l^+(v) - s_\l^+ - D^{\frac{\l-1}{r-1}-\d}f_\l \;\;\;\;\;\;\;\;\;\; \text{ for }  2 \le \l \le r-1\\
%$$m_{\l, lower}^+(v) := d_\l^+(v) - s_\l^+ + D^{\frac{\l-1}{r-1}-\d}f_\l   \;\;\;\;\;\;\;\;\;\; 2 \le \l \le r-1$$
Z_{\l}^-(v) &:= d_\l^-(v) - s_\l^- - D^{\frac{\l-1}{r-1}-\d}f_\l \;\;\;\;\;\;\;\;\;\; \text{ for } 2 \le \l \le r
%$$m_{\l, lower}^-(v) := d_\l^-(v) - s_\l^- +D^{\frac{\l-1}{r-1}-\d}f_\l \;\;\;\;\;\;\;\;\;\; 2 \le \l \le r$$
%$$m_{q, lower}:= Q - nq + nD^{-\d}f_v $$
\end{align*}
We establish the upper bound on $ V(i) $ in (\ref{eq:points}) by showing that $ Z_V < 0 $ for all 
$ i \le T $ with high probability.  Similarly, we establish the upper bounds on $ d_\l^\pm(v)$ in (\ref{eq:vertexdegree}) by showing that $ Z_\l^\pm(v) < 0 $ for all 
$ i \le T $ with high probability.   The lower bounds follow from the consideration of analogous random variables.
%Note that $d_r(v)$ does not have any positive contributions, which is why we do not have any variables corresponding to $d_r^+ (v)$.Each of the above variables is designed to enforce a bound. The ones with a subscript ``upper" (``lower") enforce upper (lower) bounds, and we will see that they are super-(sub-)martingales. We will then apply deviation inequalities to show that w.h.p. these variables stay negative (positive). We only show the calculations for the upper bounds, as the lower bounds are symmetric. 
%
%\subsection{Showing that the $m$'s are martingales}
%

We begin by showing that the sequences $ Z_V$ and $ Z_\l^\pm $ are supermartingales.  We will see that each of these calculations imposes
a condition (inequality) on the collection of error functions $ \{ f_v \} \cup \{f_\ell \mid  \ell =2, \dots, r \}$ and their derivatives.  These differential equations are 
the {\em variation equations}.  We choose error functions that satisfy the variation equations after completing the expected change calculations.  The functions will be chosen so that all error functions evaluate to 
$1$ at $t=0$ and are increasing in $t$.
After we establish that the sequences are indeed supermartingales, we use the fact that they have initial values that are negative and relatively large in absolute value.  We complete the proof by applying martingale deviation inequalities to show that it is very unlikely for these supermartingales to ever be positive.

We start the martingale calculations with the variable $ Z_V$.  We treat this 
first case in some detail in an effort to illuminate our methods for the reader 
who is not familiar with these techniques.  
Let $ S_t = D^{\frac{1}{r-1}} /N $
and recall that $ t = i/S_t$.  (The quantity $S_t$ is sometimes called the {\em time scaling}.)
We write
\begin{multline*}
%\label{eq:extra}
\Delta Z_V  \ =  \ \big( |V(i+1)| - |V(i)| \big)  \ -  \ N \big( q(t+1/S_t) - q(t) \big) \\ - \ ND^{-\d} \big( f_v( t + 1/S_t) - f_v(t) \big),
\end{multline*}
and make use of the estimates
\begin{align*}
q(t+ 1/S_t) - q(t) & = \frac{ q'(t) }{ S_t} + O \left(  \frac{q''}{ S_t^2} \right) \\
f_v(t+1/S_t) - f_v(t) & = \frac{ f'_v(t)}{ S_t} + O \left(   \frac{ f_v''}{ S_t^2} \right)
\end{align*}
where $ q'' $ and $ f_v'' $ are understood to be bounds on the second derivative that hold uniformly in the interval of interest.  We will see that the main terms of $ \Delta |V(i)| $ and $ N \left( q(t+1/S_t) - q(t) \right) $ cancel exactly.  The second order terms from $ \Delta |V(i)| $  
will be then be
balanced by the $ ND^{-\d} \left( f_v( t + 1/S_t) - f_v(t) \right) $ term.  We now proceed with
the explicit details.

 Note that $ q' = - s_2 D^{ - \frac{1}{ r-1}}$ and that $q''$ is $q$ times a polynomial in 
$t$, and recall
$ N = \Omega \left( D^{\frac{1}{r-1} + \epsilon} \right)$ (see (\ref{eq:notdense})).  We have
\begin{equation*}
\begin{split}
\E\left[\Delta Z_{V} \right] & =- \frac{1}{ |V(i)|}  \left[\sum_{v \in V(i)} (d_2(v) + 1) \right] + s_2 - D^{\frac{1}{r-1} - \d }f_v' \\
& \hskip2cm  +  \tilde{O} \left( \frac{D^{\frac{2}{r-1}  }}{N} q \right) +  O \left(\frac{D^{\frac{2}{r-1} - \d}}{N} f_v '' \right)\\
& \le D^{\frac{1}{r-1} - \d } \left[ 2f_2 - f_v ' \right] + O \left(D^{\frac{1}{r-1} - \d - \e} f_v '' +  D^{ \frac{1}{r-1} - \e + o(1)} \right)
\end{split}
\end{equation*}
whence we derive the first variation equation:
\begin{equation}\label{vareq00}
f_v ' > 2f_2.
\end{equation}
While a condition along the lines the of (\ref{vareq00}) suffices, we impose a stronger
condition in order to simplify our calculations.  We will choose
our error functions so that we have
\begin{equation}\label{vareq0}
f_v ' > 3f_2.
\end{equation}
So long as (\ref{vareq0}) holds and $ f_v'' = o( D^{\e} )$, the sequence 
$Z_V$ is a supermartingale (note that we apply the fact that $ f_2 >1$ here).  
We address the condition on $ f_v''$ below.

Now we turn to $ Z_\l^+$.  (The reader familiar with the original analysis of the $H$-free process \cite{BK} should note that there is no `creation fidelity' term here as, thanks to the convention that removes any edge that contains another edge,
selection of a vertex in an edge $e$ cannot close another vertex in the same edge.)  
For $2 \le \l \le r-1$ we have
\begin{equation*}
\begin{split}
\E\left[\D Z_{\l}^+(v) \right] & = \frac{\l d_{\l+1}(v)}{ |V(i)|} - \frac{\l s_{\l+1}}{Nq} - \frac{D^{\frac{\l}{r-1} - \d}}{N} f_\l '  \\
& \hskip2cm  +  \tilde{O} \left( \frac{D^{\frac{\l+1}{r-1}  }}{N^2}\right) + O \left(\frac{D^{\frac{\l+1}{r-1} - \d}}{N^2} f_\l ''\right)\\
& \le \frac{\l \left(s_{\l+1} + 2D^{\frac{\l}{r-1} - \d}f_{\l+1} \right)}{Nq-ND^{-\d}f_v} - \frac{\l s_{\l+1}}{Nq} - \frac{D^{\frac{\l}{r-1} - \d}}{N} f_\l ' \\
& \hskip2cm +  \tilde{O} \left( \frac{D^{\frac{\l}{r-1} -\e  }}{N}\right) + O \left( \frac{ D^{\frac{\l}{r-1} - \d - \e}}{N} f_\l '' \right) \\
& \le \frac{D^{\frac{\l}{r-1} - \d}}{N} \cdot \left[ 2 \l q^{-1}f_{\l+1} + \l \binom{r-1}{ \l}t^{r-\l-1}q^{\l-2}f_v - f_\l ' \right] \\
& \hskip2cm  +  \tilde{O} \left( \frac{D^{\frac{\l}{r-1} - \e }}{N}\right) +  O \left(  \frac{ D^{\frac{\l}{r-1} - \d - \e}}{N} f_\l '' \right)
\end{split}
\end{equation*}
(note on the second line we use $s_{\l+1} = s_{\l+1}^+ - s_{\l+1}^-$) whence we derive the following variation equations for $2 \le \l \le r-1$:
\begin{equation}\label{vareq1}
f_\l ' >  5 \l q^{-1}f_{\l+1}
\end{equation}
\begin{equation}\label{vareq2}
f_\l ' >  2\l \binom{r-1}{ \l}t^{r-\l-1}q^{\l-2}f_v
\end{equation}
So long as \eqref{vareq1}, \eqref{vareq2} hold, $ \d < \e$ and $ f_\l'' = o( D^{\e}) $ the sequence
 $Z_{\l}^+(v)$ is a supermartingale.

Finally, we consider $ Z_\ell^-$ for $ 2 \le \l \le r$.  The main term in the expected change 
of $ d_\l^-(v) $ comes from the selection of vertices $y$ for which there exists a vertex $x$ such that $ \{y,x\} \in {\mathcal H}(i) $ and there is 
an edge $e$ counted by $d_\l(v) $ such that $ x \in e$.  However, here we must also 
account for the convention of removing redundant edges.
For a fixed edge $e$ counted by $ d_\l(v)$, the selection of any vertex in 
the following sets results in the
removal of $e$ from this count:
\[ \{  y \in V(i) :  \exists A \subset e  \text{ such that } A \neq \{v\} \text{ and } A \cup \{y\} \in \mc{H}(i)  \}. \]
(Note that this is essentially restating the convention that we remove edges that contain other edges). 
%In other words, the set 
%\[\{B\setminus A: A \subseteq e, B \in d_{A \up |A|+1} \} \setminus d_2(v)\]
Together, the sums \[\sum_{x \in e \setminus\{v\}} d_2(x) + \sum_{A \subseteq e, |A|\ge 2} d_{A \up |A|+1} \] count each $y$ with the property that the choice of $y$ causes the removal of $e$ from the count $ d_\l(v)$ at least once and at most $O(1)$ many times. 
%No $y$ gets counted by both sums. 
The number of $y$ that are counted more than once in the first sum is at most $ \binom{ \l-1}{2} C_{2,2 \to 1}$. Therefore we get the following estimate: \[ \E\left[\D d_\l^-(v) \right] =  
\frac{1}{|V|} \left\{ \sum_{e \in d_\l(v)} \sum_{u \in e \setminus\{v\}} d_2(u) + O\left( d_\l \cdot \left[C_{2,2 \to 1} + \sum_{k = 2}^{\l -1} D_{k \up k+1} \right] \right) \right\} \]
(We note in passing that this estimate also takes into account the selection of vertices that
cause $v$ itself to be removed from the vertex set.  In other words, this estimate also takes freezing into account.)
%Thus, the following accounts for all but a negligible portion of the above set.
%\[  \{ y \in V(i) : \exists x \in e \setminus \{v\} \text{ such that } \{x,y\} \in \mc{H}(i) \}  \]
%The rest of the set can be bounded in terms of our higher-order degrees and codegrees.
 We have
\begin{equation*}
\begin{split}
\E\left[\D Z_\l^-(v) \right]
& = \frac{1}{|V(i)|} \left\{ \sum_{e \in d_\l(v)} \sum_{u \in e \setminus \{v\}} d_2(u)  + O
\left( d_\l \cdot \left[C_{2,2 \to 1} + \sum_{k = 2}^{\l -1} D_{k \up k+1} 
\right] \right) \right\}\\
& \;\;\;\;\;\;\;\;\;\; - \frac{(\l-1)s_\l \cdot s_2}{Nq} - \frac{D^{\frac{\l
}{r-1} - \d}}{N} f_\l ' + \tilde{O} \left( \frac{D^{\frac{\l+1}{r-1}  } }{N^2} \right)+   O \left(\frac{D^{\frac{\l+1}{r-1} - \d}}{N^2} f_\l 
'' \right)\\
&\le \frac{(\l-1)\left( s_\l + 2D^{\frac{\l-1}{r-1} - \d}f_\l \right)\left( s
_2 + 2D^{\frac{1}{r-1} - \d}f_2 \right)}{Nq-ND^{-\d}f_v}  - \frac{(\l-1)s_\l 
\cdot s_2}{Nq} - \frac{D^{\frac{\l}{r-1} - \d}}{N} f_\l ' \\
&\;\;\;\;\;\;\;\;\;\;+ \tilde{O} \left( \frac{D^{\frac{\l+1}{r-1}  } }{N^2} \right) + O \left( \frac{D^{\frac{\l}{r-1} - \frac{\e}{2}} }{N} q^{\l-2}   + \frac{D^{\frac{\l+1}{r-1} - \d}}{N^2} f_\l ''  \right)\\
& \le \frac{D^{\frac{\l}{r-1} - \d}}{N} \cdot \left[ (2 + o(1))(\l-1)\binom{r
-1}{ \l -1}t^{r-\l}q^{\l -2}f_2 + 2(\l -1)(r-1)t^{r-2}f_{\l} \right.\\
& 
\;\;\;\;\;\;\;\;\;\;\;\;\;\;\;\;\;\;\;\;\;\;\;\;\;\;\;\;\;\;\;\;\;\;\;\;\;\;  
\ \ \ \ \ \ \ \ \ \
\left. + (\l -1)(r-1)\binom{r-1}{ \l -1}t^{2r- \l -2}q^{\l-2}f_v  - f_\l ' 
\right]\\
&\;\;\;\;\;\;\;\;\;\;+  \tilde{O} \left( \frac{D^{\frac{\l}{r-1} -\e} }{N} \right)  + O \left( \frac{D^{\frac{\l}{r-1} - \frac{\e}{2}}
}{N} q^{\l-2}   + \frac{D^{\frac{\l}{r-1} - \d - \e}}{N} f_\l '' \right)
\end{split}
\end{equation*}
whence we derive the following variation equations for $2 \le \l \le r$:
\begin{equation}\label{vareq3}
f_\l ' >  7(\l-1)\binom{r-1}{ \l -1}t^{r-\l}q^{\l -2}f_2
\end{equation}
\begin{equation}\label{vareq4}
f_\l ' >  6(\l -1)(r-1)t^{r-2}f_{\l}
\end{equation}
\begin{equation}\label{vareq5}
 f_\l ' >  3(\l -1)(r-1)\binom{r-1}{ \l -1}t^{2r- \l -2}q^{\l-2}f_v
\end{equation}
So long as \eqref{vareq3}, \eqref{vareq4}, \eqref{vareq5} hold and $ \e/2 > \d $ and $ f_\l'' = o(D^\e)$ the sequence  $Z_{\l}^-(v)$ is a supermartingale.  (We will see below that $ q^\l f_2 > 1 $.)

We satisfy the variation equations \eqref{vareq0}, \eqref{vareq1}, \eqref{vareq2}, \eqref{vareq3}, \eqref{vareq4}, \eqref{vareq5} by setting the error functions to have the form $$f_\l =  \left( 1 + t^{r - \l + 2} \right)   \cdot \exp \left(\a t + \b t^{r-1}\right) \cdot q^\l $$  $$f_v =  \left( 1 + t^{2} \right)  \cdot \exp \left(\a t + \b t^{r-1}\right) \cdot q^{2}$$ for some constants $\a$ and $ \b$ depending only on $r$.  Note that (dropping some terms) we have 
$$f_\l ' \ge \left[ \a   + (\b-\l)(r-1) t^{2r-\l}  \right]  \cdot \exp \left(\a t + \b t^{r-1}\right) \cdot q^\l$$
$$f_v ' \ge \left[ \a  + (\b-2)(r-1) t^{r} \right] \cdot \exp \left(\a t + \b t^{r-1}\right) \cdot q^2.$$
Note that for this choice of functions, all variation equations have the property that both sides of the equation
have the same exponential term.  It remains to compare the polynomial terms; in each case it is clear that we get the
desired inequality by choosing $ \a $ and $ \beta $ to be sufficiently large (as functions of $r$).  We get the desired conditions on second derivatives by choosing $\z$ sufficiently small (recall that we are free to choose $ \z$ arbitrarily small).

We complete the proof by applying martingale variation inequalities to prove that $ Z_V $ and $ Z^\pm_\ell $ remain negative with high probability.  We will apply the following lemmas (which both follow from Hoeffding \cite{Hoef}):

\begin{lemma}\label{symmetricAH} Let $X_i$ be a supermartingale such that 
$|\Delta X| \leq c_i$ for all $i$. Then $$ \P(X_m - X_0 > d) \leq  \exp\left(-\frac{d^2}{2 \displaystyle\sum_{i\leq m} c_i^2 }\right)$$ \end{lemma}

\begin{lemma}\label{asymmetricAH} Let $X_i$ be a supermartingale such that $-N \leq \Delta X \leq \eta$ for all $i$, for some $\eta < \frac{N}{10}$. Then for any $d < \eta m$ 
we have $$ \P(X_m -X_0 > d) \leq \exp \left(- \frac{d^2}{3 m \eta N }\right)$$
\end{lemma}

For our upper bound on $|V(i)|$ we apply Lemma~\ref{symmetricAH} to the supermartingale $Z_{V}(i)$. Note that if the vertex $v$ is inserted to the independent set at step $i$ then $\D V = -1-d_2(v) \in s_2 \pm D^{\frac{1}{r-1}-\d}f_2$.  As $q' = - D^{- \frac{1}{r-1}} s_2$  we have 
\begin{multline*} 
\D Z_{V} = O\left( D^{\frac{1}{r-1}-\d} (f_2+f_v ') \right)  +  O \left( \frac{ D^{ \frac{2}{r-1} - \d}}{N} f_v'' \right) + \tilde{O} \left( \frac{ D^{\frac{2}{r-1}}}{N} \right) \\ = O\left( D^{\frac{1}{r-1}-\d} (f_2+f_v ') \right) .
\end{multline*}  
As we have $Z_V(0) = - ND^{-\d} $,  the probability that $Z_V$ is positive at step $T$ is at most
\begin{multline*}
\exp\left\{- \tilde\Om \left( \frac{\left(ND^{-\d} \right)^2}{ND^{-\frac{1}{r-1}} \cdot \left[D^{\frac{1}{r-1}-\d} (f_2+f_v')\right]^2} \right) \right\} 
\le \exp\left\{ - \tilde\Om \left(   D^{\e} (f_2+f_v')^{-2}  \right) \right\} \\
\le \exp\left\{ -N^{\Om(1)} \right\},
\end{multline*}
so long as $ \z$ is sufficiently small. Note that we have used the notation $\tilde{\Omega}(\cdot)$: If $f$ and $g$ are functions of $N$ such that $f$ is bounded above by $g$ times some poly-logarithmic factor we write $ f = \tilde{\Omega} (g) $. Also note that $D > N^\epsilon$ (see (\ref{eq:notdense})) is used to get the last expression above.

For our bound on $d_\l^+(v)$ we apply Lemma~\ref{asymmetricAH} to the supermartingale $Z_{\l}^+(v)$. Note that we have  $$ \D Z_{\l}^+(v)  < D_{2 \up \l+1} \le  D^{\frac{\l-1}{r-1} - \frac{\e}{2}} $$ and note that the upper bound is not affected by the convention of removing edges containing other edges, since that would only decrease the number of new edges. For a lower bound we have $$\D Z_{\l}^+(v) > -(s_\l^+)'\cdot \frac{D^\frac{1}{r-1}}{N} + O\of{(s_\l^+)''\cdot \frac{D^\frac{2}{r-1}}{N^2} + f_\l'\cdot \frac{D^\frac{1}{r-1}}{N}} > -O\left( \frac{D^\frac{\l}{r-1}}{N} \right).$$    We have $Z_{\l}^+(0) = - D^{\frac{\l-1}{r-1}-\d}$, and the hypotheses of Lemma~\ref{asymmetricAH} hold since $\frac{D^\frac{\l}{r-1}}{N} = o(D^{\frac{\l-1}{r-1} - \frac{\e}{2}})$ and $D^{\frac{\l-1}{r-1}-\d} =o(\frac{D^\frac{\l}{r-1}}{N} \cdot i_{max} ) $. So the probability that $Z_{\l}^+(v)$ is positive at some step $i \leq T$ is at most
\begin{multline*}
\exp\left\{-\tilde\Om \left( \frac{\left(D^{\frac{\l-1}{r-1}-\d} \right)^2}{ND^{-\frac{1}{r-1}} \cdot \frac{1}{N}D^\frac{\l}{r-1} \cdot D^{\frac{\l-1}{r-1} - \e/2} } \right) \right\} \\ \le \exp \left\{ - \tilde\Om \left( D^{\frac{\e}{2}-2\d}\right) \right\} \le \exp\left\{ -N^{\Om(1)} \right\}.
\end{multline*}
Note that we have use $ \d < \e/4 $ to obtain the last expression.

For our bound on $d_\l^-(v)$ we apply Lemma~\ref{asymmetricAH} to the supermartingale $Z_{\l}^-$. Note that $$-O\left( \frac{D^\frac{\l}{r-1}}{N} \right)< \D Z_{\l}^-(v)  < \displaystyle  O \left( \sum_{1 \le k \le \l-1} C_{\l, k+1 \to k} \right) = O\left( D^{\frac{\l-1}{r-1} - \frac{\e}{2}} \right) $$ (and note here that the upper bound includes edges lost because they contain other edges). Thus, the rest of the calculation is the same as it was for $d_\l^+(v)$.

%\section{Choosing the constants}

%

%Recall that $\lambda = \frac{\z^{r-1}}{\e}, \;\;\;\; \k = \e - 2r \lambda, \;\;\;\; \b = 1+25r \cdot 2^r$. Now if we choose some $$\d < \frac{\e}{(r-1)^2},$$ and $$\z < min \left\{ \left(\frac{\e^2}{4r}  \right)^\frac{1}{r-1}, \left(\frac{\d}{\b}  \right)^\frac{1}{r-1}  \right\}$$ then \eqref{constcondition1} and \eqref{constcondition2} are both satisfied. 

%

%Note that since $\z < \left(\frac{\d}{\b}  \right)^\frac{1}{r-1} $, we have that $$Q \ge nq - nD^{-\d}f_v \ge \frac{1}{2} n^{1- \z^{r-1}} \ge \frac{1}{2} n D^{- \lambda},$$ as promised in section \ref{sec:stoppingtime}

\section{Subgraph counts: Proof of Theorem~\ref{theory:count}}

\label{sec:count}

%

%Set $p:= \frac{i_{end}}{n} = \z D^{-\frac{1}{r-1}} \log^\frac{1}{r-1} n$. Let $\mc{G}$ be an $s$-uniform hypergraph with ordered edges (so an edge $e$ is $\{v_{1,e} \ldots v_{s,e} \}$), such that no edge of $\mc{G}$ contains an edge of $\mc{H}$, and for $1 \le a \le s$, any set of $a$ vertices is contained in at most $\D_a = o\left(|\mc{G}|D^{-\frac{a}{r-1}} \right)$ edges of $\mc{G}$. Suppose $R \subset \{1 \ldots s\}$ and  $|\mc{G}|p^{|R|} \rightarrow \infty$. 

%

% Let the random variable $X_R$ be the number of edges $e \in \mc{G}$ such that $e \cap S(i_{end}) = \{v_{r, e}: r \in R\}$. We claim that $$E[X_R] = |\mc{G}| p^{|R|} \cdot (1+o(1))$$ and that $X_R$ is concentrated.

Here we apply the observation, due to Wolfovitz~\cite{Wz3}, that the classical second moment 
argument for subgraph counts can be applied in the context of the random greedy independent set process.
\begin{lemma}
\label{lem:wolf}
Fix a constant $L$ and suppose  $\{v_1 \ldots v_{L} \} \subset V$ does not contain 
an edge of $\mc{H}$. Then for all $ j \le i_{\rm max} $ we have
\[ \P \left( \{v_1 \ldots v_{L} \} \subset I(j) \right) = (j/N)^L \cdot (1+o(1)). \]
\end{lemma}
\begin{proof} Fix a permutation of this set of vertices, say $u_1 \ldots u_{L}$ after relabeling, and a list of 
steps of the algorithm $i_1< \dotsm < i_{L} \le j$.  Let  $\mc{E}$ be the event that each $u_k$ is chosen on step $i_k$ for $ k =1, \dots, L$.  Note that the event $ \mc{E}$ requires that vertex $ u_k$ remains in $V(i)$ until step $i_k-1$, and, in order to achieve this condition, the set $\{v_1 \ldots v_L\}$ can never contain an edge of $ \mc{H}(i) $.

Let $ \mc{E}_i $ be the event that  $ T > i $ and 
the first $i$ steps of the algorithm are compatible with $ \mc{E} $.  Then
we write
\[ \P( \mc{E}_1) \prod_{i=2}^{i_L} \P \left( \mc{E}_i \mid \mc{E}_{i-1} \right) \le  \P( \mc{E} ) \le  \P( \mc{E}_1) \prod_{j=2}^{i_L} \P \left( \mc{E}_i \mid \mc{E}_{i-1} \right) + \P( T \le i_L). \]
If $i = i_k$ then, conditioning on the first $ i-1$ steps of the algorithm and the event $ \mc{E}_{i-1}$, we have 
$ \P( \mc{E}_i ) = \frac{1}{|V(i-1)|} $, unless the selection of $u_k$ triggers the stopping time $T$. Also recall that $\P( T \le i_L)\le \exp \left\{ -N^{\Om(1)}\right\}$, and since $ \P( \mc{E}_i ) =  N^{-O(1)}$ we have $\P( T =i \mid \mc{E}_{i-1})\le \exp \left\{ -N^{\Om(1)}\right\}$ for every choice of $i, u_1 \ldots u_L$. Thus, we can write
\[ \P( \mc{E}_i \mid \mc{E}_{i-1} ) = \frac{1}{ Nq( 1 \pm D^{ -\d/2})} \pm \exp \left\{ -N^{\Om(1)} \right\}
= \frac{ (1+o(1))}{ Nq }. \]

If $i_k < i < i_{k+1}$, then $ \P ( \mc{E}_i \mid \mc{E}_{i-1})$  is 
the probability that the set of vertices $\{u_{k+1}, \ldots, u_L\}$ 
all stay open and do not obtain an edge and we do not trigger the stopping time $T$. That is, we have
\begin{eqnarray*}
\P( \mc{E}_i \mid \mc{E}_{i-1}) & = & 1-  \frac{1}{Q(i-1)} \left[ \sum_{m=k+1}^L d_2(u_{w}) + O \left(C_{2,2 \to 1} + \displaystyle \sum_{m \ge 2} D_{m \up m+1} \right) \right] \pm \exp \left\{ -N^{\Om(1)}\right\}\\
& = &  1- \frac{(L+1 - k)(r-1)D^{\frac{1}{r-1}}t^{r-2}q \pm O\left(D^{\frac{1}{r-1} - \frac{\e}{2}} \right)}{Nq \cdot \left(1 \pm D^{ - \d/2 } \right)}\pm \exp \left\{ -N^{\Om(1)}\right\}\\
& = & 1-  \frac{ (L+1 - k)(r-1) D^{\frac{1}{r-1}}t^{r-2}}{N} \pm O \left(D^{\frac{1}{r-1} - \frac{\d}{4}}  N^{-1} \right) 
\end{eqnarray*}
where on the last line we use $\zeta \ll \d \ll \e$ and \eqref{eq:minq}. Thus, setting $i_0=0$ we have
\begin{eqnarray*}
\P(\mc{E})
& = & \prod_{k=1}^{L} \left[ \prod_{i=i_{k-1}}^{i_k-2} 1-  \frac{ (L+1 - k)(r-1) D^{\frac{1}{r-1}}t^{r-2}}{N} \pm O \left(D^{\frac{1}{r-1} - \frac{\d}{4}}  N^{-1} \right) \right] \frac{1 + o(1)}{Nq(t(i_k-1)) }\\
& = & (1+o(1)) \exp \left\{  - \frac{ (r-1)D^{\frac{1}{r-1}} }{N}\sum_{k=1}^{L}(L+1 - k) \sum_{i=i_{k-1}}^{i_k-2}t^{r-2}   \right\}\prod_{k=1}^{L} \frac{1 } {Nq(t(i_k-1))}\\
& = & (1+o(1)) \exp \left\{  -\frac{ D^{\frac{1}{r-1}} }{N}\sum_{k=1}^{L} \sum_{i=0}^{i_k-2} (r-1) t^{r-2}  \right\}\prod_{k=1}^{L} \frac{1 } {Nq(t(i_k-1))}\\
& = & (1+o(1)) \exp \left\{ - \sum_{k=1}^{L} t(i_{k}-1)^{r-1}  + O\left(\frac{t(i_{max})^{r-2} D^\frac{1}{r-1}}{N} \right) \right\}\prod_{k=1}^{L} \frac{1 } {Nq(t(i_k-1))}\\
& = & (1+o(1)) \frac{1}{N^L}
\end{eqnarray*}
where on the last line we have used $q(t) = \exp\{-t^{r-1}\} $.

We complete the proof by summing over all possible choices of the indices $i_k$.
\end{proof}

%

%A simple inclusion-exclusion calculation gives the following corollary:

%

%\begin{corollary}

%Fix constants $K \le L$ and suppose  $\{v_1 \ldots v_{L} \}$ is a set of vertices not containing an edge of $\mc{H}$. Then %the probability that $\{v_1 \ldots v_{L} \} \cap S(i_{end}) = \{v_1 \ldots v_{K} \}$ is $p^K \cdot (1+o(1))$.%

%\end{corollary}

Now by linearity of expectation, we have $E[X_\mc{G}] = |\mc{G}| p^{s} \cdot (1+o(1))$. Now we will do a second moment calculation to show that $X_\mc{G}$ is concentrated around its mean. It suffices to show that 
$ E[X_\mc{G}^2] \le E[X_\mc{G}]^2 \cdot(1+o(1)) $.

% The second moment calculation will also rely on the lemma. 

%

%First we write $X_R$ as the sum of indicators: $$X_R = \sum_{e \in \mc{G}} 1_{e \cap S = \{v_{r,e}: r \in R\}}.$$ Thus, $$X_R^2 = X_R + \sum_{e, e' \in \mc{G}} 1_{e \cap S = \{v_{r,e}: r \in R\} \; \wedge \; e' \cap S = \{v_{r,e'}: r \in R\}}$$ 

We have
$$E[X_\mc{G}^2] = \sum_{e, e' \in \mc{G}} \P (e \cup e' \subseteq I(i) ).$$
%Note that the number of pairs $ e,e'$ of disjoint edges of $ \mc{G} $ such that $ e \cup e'$ contains an
%edge of $ \mc{H} $ is at most
%\[ |\mc{G}| \sum_{a =r/2 }^{r-1 \wedge s} \Delta_a( \mc{H} ) \Delta_{r-a}( \mc{G}) = 
%o \left( |\mc{G}| \cdot D^{ \frac{r-a}{r-1} - \e} \cdot |\mc{G}| p^{r-a} \right) = o( |\mc{G}|^2 )  \]
Thus, by an application of the Lemma, we have
\begin{eqnarray*}
E[X_\mc{G}^2] & = & \sum_{e \in \mc{G}} \sum_{a=0}^s | \{ e' \in \mc{G} : |e \cap e'|=a \} | p^{2s-a} (1+o(1)) \\
& \le & |\mc{G}|^2 p^{2s}(1+o(1)) +  O \left( |\mc{G}| \sum_{a=1}^s \Delta_a( \mc{G} ) p^{2s-a} \right) \\
& = & (1+o(1)) E[ X_\mc{G} ]^2.
\end{eqnarray*}
To help see the second line, note that the $1+o(1)$ from the line above is unaffected by the sums since it is uniform over the terms. To see the last line note that $|\mc{G}| \Delta_a( \mc{G} ) p^{2s-a} = o(|\mc{G}|^2 p^{2s})$ for each $a$. 

%Some pairs $ e,e' $ may have the property that $ e \cup e'$ contains an edge of $ \mc{H}$.

%

%We will see that, of the terms on the right side, 

%\begin{eqnarray}

%&&\sum_{|e \cap e'| =0} P[e \cap S = \{v_{r,e}: r \in R\} \; \wedge \; e' \cap S = \{v_{r,e'}: r \in R\}] \nonumber\\

%&& = E[X_R]^2 \cdot (1+o(1)) \label{2ndmoment}

%\end{eqnarray} and the rest is negligible. Note that of the $|\mc{G}|^2$ pairs $(e, e')$, there are only $O\left(|\mc{G}| \cdot D_{k \up r} \cdot \D_k \right) = o\left(|\mc{G}|^2 \right)$ many such that $e \cup e'$ contains an $\mc{H}$-edge. Now an application of the lemma gives \eqref{2ndmoment}

%Now for $k \ge 1$  the sum $\sum_{|e \cap e'| =k} P[\{v_{r,e}: r \in R\} \cup \{v_{r,e'}: r \in R\} \subset S]$ has at most $\binom{s}{k}|\mc{G}| \cdot \D_k = o\left( |\mc{G}|^2 p^k\right)$ many terms, and each term is at most $p^{2|R|-k} (1+o(1))$. Thus the entire sum is $o\left(E[X_R]^2 \right)$ and we're done.

\noindent

{\bf Acknowledgment.}  The authors thank Jacob Fox, Dhruv Mubayi, Mike Picollelli and Jozsef Balogh for
helpful conversations.  We also thank the anonymous referees for many helpful comments. We also thank Asaf Ferber, Dhruv Mubayi,
Gwen McKinley, Wojtek Samotij, Nicholas Spanier and Lutz Warnke for identifying
omissions in earlier versions of this work.

\end{document}